\documentclass[12pt]{article}

\usepackage{amsmath,amsthm,amsfonts,amssymb,amscd,caption,color,subcaption,cite}
\usepackage{graphicx}
\usepackage[mathscr]{eucal}
\allowdisplaybreaks[4]
\usepackage[figurename=Fig.]{caption}
\numberwithin{equation}{section}
\allowdisplaybreaks[4]

\newtheorem {Lemma}{Lemma}[section]
\newtheorem {Theorem} {Theorem}[section]

\newtheorem {Claim} {Claim}[section]
\usepackage{tikz}
\usetikzlibrary{positioning}
\usepackage{xcolor}
\usepackage{circuitikz}
\usetikzlibrary{patterns}
\usepackage{fullpage}

\begin{document}

\title{The spectral radius of unbalanced signed graphs without negative $C_3$ or $C_4$}

\author{Yiting Cai\footnote{E-mail: yitingcai@m.scnu.edu.cn}, 
Bo Zhou\footnote{E-mail: zhoubo@scnu.edu.cn}\\
School of  Mathematical Sciences, South China Normal University,\\
Guangzhou 510631, P.R. China}

\date{}
\maketitle

\begin{abstract}
A signed graph is a graph in which every edge carries a $+$ or a $-$ sign.
%A signed graph is unbalanced if it has a negative cycle.
In this paper, we determine the signed graphs with maximum spectral radius among all unbalanced signed graphs with fixed order  that contain  neither negative three-cycles nor negative  four-cycles. \\ \\
{\bf Key words:} unbalanced signed graph, negative cycles, spectral radius\\ \\
{AMS Classification:} 05C50
\end{abstract}

\section{Introduction}

A signed graph $\Gamma$ is a pair $\Gamma= (G, \sigma)$ where $G$ is a simple graph, called the
underlying graph,  with vertex set $V(G)$ and edge set $E(G)$,  and  $\sigma : E(G) \rightarrow \{-1, 1\}$ is the sign function, called  the signature.
The vertex (resp. edge) set of $\Gamma$ is just the vertex (resp. edge) set of its underlying graph,
denoted by $V(\Gamma)$ (resp. $E(\Gamma)$).
An edge $e$ is positive
(resp. negative) if $\sigma(e) = 1$ (resp. $\sigma(e) = -1$).
If all edges of $\Gamma=(G, \sigma)$
are positive (resp. negative), then we denote it by $(G,+)$ (resp. $(G,-)$). $-\Gamma$ denotes the signed graph
obtained from $\Gamma$ by reversing the signs of all edges.
%A graph may be viewed as a signed graph with the all positive signature, that is the signature which assigns $1$ to every edge.
The adjacency matrix $A(\Gamma)$
of $\Gamma$ is obtained from the standard adjacency matrix of the underlying graph $G$ by reversing the sign of all $1$s which correspond to negative edges. That is, $A(\Gamma) = (a_{uv}^\sigma)_{u,v\in V(\Gamma)}$  with $a_{uv}^\sigma= \sigma(uv)$ if $u$ and $v$ are adjacent, and $0$ otherwise.
The eigenvalues of $A(\Gamma)$ are the eigenvalues of $\Gamma$, which are denoted by $\lambda_1(\Gamma)\ge \dots\ge \lambda_n(\Gamma)$, where $n$ is the order of $\Gamma$.
The spectral radius $\rho(\Gamma)$ of $\Gamma$ is the maximum absolute values of the eigenvalues of $\Gamma$. That is,
$\rho(\Gamma)=\max\{\lambda_1(\Gamma), -\lambda_n(\Gamma) \}$. If $\Gamma=(G, +)$, then $\rho(\Gamma)=\lambda_1(\Gamma)$ is the spectral radius $\rho(G)$ of $G$.

For a signed graph $\Gamma$ with $\emptyset\ne U\subset V(\Gamma)$.  Let $\Gamma^U$ be the signed graph obtained from $\Gamma$ by reversing the sign of each edge between a vertex in $U$ and a vertex in $V(\Gamma)\setminus U$.
The signed graphs  $\Gamma$ and $\Gamma^U$ are
said to be switching equivalent.
The switching equivalence is an equivalence relation that
preserves the eigenvalues.
A cycle in $\Gamma$ is said to be positive (or balanced) if it contains an even number of negative edges,
otherwise the cycle is  negative (or unbalanced).
%Note that the signature
%switching does not affect the sign of the cycles.
$\Gamma=(G, \sigma)$ is said to be balanced if it is switching equivalent to $(G,+)$.
Otherwise, it is said to be unbalanced. Equivalently, $\Gamma$
is balanced if every cycle is positive, see \cite{Za}.
From \cite{St1}, for a signed graph $\Gamma= (G,\sigma)$, we have $\lambda_1(\Gamma)\le \lambda_1(G)$
with equality when it is connected if and only if $\Gamma$ is balanced. As in \cite{AB,BB}, we consider only unbalanced signed graphs when we study the extremal problems about the spectral radius or the largest eigenvalue.

%An isomorphism of two signed graphs is an isomorphism of underlying graphs that preserves edge signs.
Two signed graphs are called switching isomorphic if one of them is  a switching equivalent to an  isomorphic copy of  the other.

Given a set $\mathcal{F}$ of signed graphs, if a signed graph $\Gamma$ contains no signed subgraph isomorphic to any one in $\mathcal{F}$, then $\Gamma$ is called $\mathcal{F}$-free.

Denote by $\mathcal{K}_n^-$  the set of unbalanced signed complete graphs of order $n$, and for $n\ge 3$, $\mathcal{C}_n^-$ the set of unbalanced signed cycles of order $n$.

For a graph $G$ with an edge $e$, $G-e$ denotes the graph obtained from $G$ by removing the edge $e$. Denote by $K_n$ the complete graph of order $n$. For $n\ge 3$, $C_n$  denotes the cycle of order $n$.

For $r\ge 3$, the $\mathcal{K}_{r}^-$-free unbalanced  signed graphs of fixed order with maximum spectral radius have been determined for $r=3$ by Wang, Hou and Li \cite{WHL},
for $r=4$ by Chen and Yuan \cite{CY}, for $r=5$ by Wang \cite{Wa}, and for $r\ge 4$ by Xiong and  Hou \cite{XH}.
The $\mathcal{C}_4^-$-free (resp. $\mathcal{C}_{2k+1}^-$-free with $3\le k\le \frac{n-11}{10}$) unbalanced signed graphs of order $n$ with maximum largest eigenvalue (resp. spectral radius) have been determined by Wang and Lin \cite{WL} (resp. Wang, Hou and Huang \cite{WHH}).

Let $\mathcal{C}_{3,4}^-=\mathcal{C}_3^-\cup \mathcal{C}_4^-$.
In this paper, we  determine  the signed graphs with maximum spectral radius among all $\mathcal{C}_{3,4}^-$-free  unbalanced signed graphs of order $n\ge 5$. For $n\ge 5$, let $\Gamma_n$ be the signed graph
obtained from a copy of $K_{n-2}$
by removing the edge $uw$ and adding two new vertices $v_1$ and $v_2$ and three edges $v_1v_2, v_1u$ and $v_2w$, where $v_1v_2$ is the unique negative edge, see Fig.~\ref{f1}.
As the shortest negative cycles in $\Gamma_n$ are $C_5$s,  $\Gamma_n$ is a $\mathcal{C}_{3,4}^-$-free   unbalanced signed graph. It is evident that a $\mathcal{C}_{3,4}^-$-free  unbalanced signed graph of order $5$ is switching equivalent to $\Gamma_5$, a negative $C_5$ with $\rho(\Gamma_5)=-\lambda_5(\Gamma_5)=2>\lambda_1(\Gamma_5)=1.6180$.

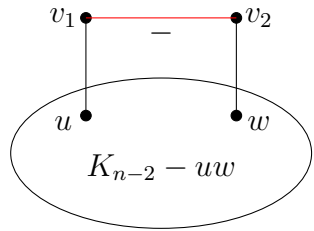
\begin{figure}[htbp]
\centering
\begin{tikzpicture}
\draw (2,0) ellipse (2 and 1);
%\draw  [black](1,0.5)--(3,0.5);
\filldraw [black] (1,0.5) circle (2pt);
\filldraw [black] (3,0.5) circle (2pt);
\node at (2, -0.2) {$K_{n-2}-uw$};
\node at (0.7, 0.4) {$u$};
\node at (3.3, 0.4) {$w$};
\node at (0.7, 1.8) {$v_1$};
\node at (3.3, 1.8) {$v_2$};
\filldraw [black] (1,1.8) circle (2pt);
\filldraw [black] (3,1.8) circle (2pt);
\draw  [black](1,0.5)--(1,1.8);
\draw  [black](3,1.8)--(3,0.5);
\draw  [red](1,1.8)--(3,1.8);
\node at (2, 1.6) {$-$};
\end{tikzpicture}
\caption{The signed graph $\Gamma_n$.}
\label{f1}
\end{figure}

%
%\begin{center}
%\begin{picture}(410,45)
%\put(160,35){\circle* {3.5}} \put(210,35){\circle* {3.5}}\put(160,35){\line(1,0) {50}}
%\put(160,0){\circle* {3.5}} \put(210,0){\circle* {3.5}}
%\put(160,35){\line(0,-1) {35}} \put(210,35){\line(0,-1) {35}}
%\put(177,40){$-1$}
%\put(156,40){$v_1$}
%\put(206,40){$v_2$}
%%\put(210,-10){$v_3$}
%%\put(153,-10){$v_4$}
%\put(160,0){\line(1,-2) {10}} \put(210,0){\line(-1,-2) {10}}
%\put(185,-25){\circle {30}}
%\put(172,-30){$K_{n-4}$}
%\put(181,-55){$\Gamma_n$}
%
%
%
%\put(100,-70) {Fig.~1. The signed graph $\Gamma_n$. }
%\end{picture}
%\end{center}
%\vspace{20mm}

\begin{Theorem}\label{N1}
Let $\Gamma$ be a $\mathcal{C}_{3,4}^-$-free unbalanced signed graph of order $n\ge 6$. Then
\[
\rho(\Gamma)\le \gamma_n
\]
with equality if and only if $\Gamma$ is switching isomorphic to $\Gamma_n$, where $\gamma_n$ is the largest root of the equation $\lambda^3-(n-6)\lambda^2-3(n-4)\lambda-n+3=0$.
\end{Theorem}

\section{Preliminaries}

Most of the concepts like connectedness and degrees of vertices defined for graphs are naturally extended to signed graphs.
Let $\Gamma=(G, \sigma)$ be a signed graph. For a vertex $u$ in $\Gamma$, we denote by $N_\Gamma(u)$ and $d_\Gamma(u)$ the neighborhood and the degree $u$ in $\Gamma$, respectively.
For $\emptyset\ne U\subset V(\Gamma)$, let $\Gamma-U$ be the signed graph obtained from $\Gamma$ be removing the vertices in $U$ and every edge incident to a vertex of $U$, and $\Gamma[U]$ denotes the signed graph
$\Gamma-(V(\Gamma)\setminus U)$. The underlying graph of $\Gamma-U$ is $G-U$.  A chord of a cycle $C$ is an edge between two non-consecutive vertices of the cycle.

Let $M$ be an $n\times n$ matrix whose rows and columns are indexed by elements in $X=\{1, \dots, n\}$.  Let $\pi=\{X_1, \dots, X_s\}$ be a partition of $X$. For $1\le i,j\le s$, let $M_{ij}$ be the submatrix of $M$ whose rows and columns are indexed
by elements of $X_i$ and $X_j$. The partition $\pi$ is  equitable if the row sum of each $M_{ij}$ is a constant for $1\le i,j\le s$.
The $s\times s$ matrix whose $(i,j)$-entry is the average row sum of $M_{ij}$ with $1\le i,j\le s$ is called the quotient matrix of $M$.

\begin{Lemma}\label{QM}\cite[Lemma 2.3.1]{BH}
Let $M$ be a real symmetric matrix. Then the spectrum of the quotient matrix of $M$ with respect to an equitable partition is contained in the spectrum of $M$.
\end{Lemma}

For integers $n$ and $\tau$ with  $n\ge 6$ and $0\le\tau\le n-5$, we define a signed graph
$\Gamma_{n, \tau}$ on vertices $v_1,\dots, v_n$ as displayed in Fig. \ref{f2}, where
$V(\Gamma_{n,\tau})=\{v_1,v_2, v_3, v_5\}\cup X\cup Y$ with $|X|=\tau$,  $|Y|=n-4-\tau$ and $v_4\in Y$,
$v_1v_2$ is the unique negative edge, $v_1\dots v_5v_1$ is an induced cycle, $d_{\Gamma_{n, \tau}}(v_1)=2+\tau$,  $d_{\Gamma_{n, \tau}}(v_2)=2$,
$\Gamma_{n, \tau}[X\cup\{v_1, v_5\}]=(K_{\tau+2}, +)$, $\Gamma_{n, \tau}[Y\cup\{v_3, v_5\}]=(K_{n-\tau-2}-v_3v_5, +)$, and
$\Gamma_{n,\tau}[X\cup Y]=(K_{n-4}, +)$.
%Here  a thick line between a vertex $v_1$ or $v_3$ (resp. $v_3$ or $v_5$) and a vertex of a set $X$ if it is not empty (resp. $Y$)  means that the vertex is adjacent to every vertex of the set.
With a proper labeling of the vertices,  $\Gamma_{n}$ is (isomorphic to) $\Gamma_{n,0}$.

Denote by $I_n$ the identity matrix of order $n$.

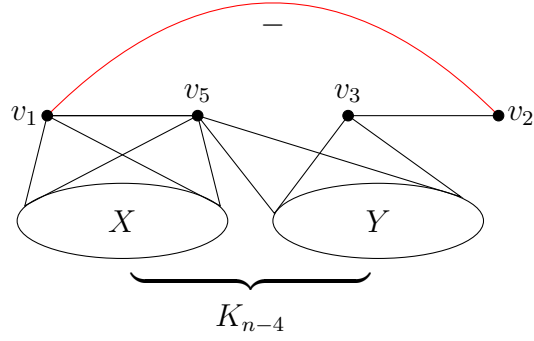
\begin{figure}[htbp]
\centering
\begin{tikzpicture}
\draw [red](0,1) .. controls (2,3) and (4,3).. (6,1);
\draw  [black](0,1)--(2,1)--(3.02,-0.3)--(4,1)--(6,1);
\filldraw [black] (0,1) circle (2pt);
\filldraw [black] (2,1) circle (2pt);
\filldraw [black] (4,1) circle (2pt);
\filldraw [black] (6,1) circle (2pt);
\node at (1, -0.4) {$X$};
\node at (4.4, -0.4) {$Y$};
\draw (1,-0.4) ellipse (1.4 and 0.5);
\draw (4.4,-0.4) ellipse (1.4 and 0.5);
\node at (2.7, -1.2) {$\underbrace{~~~~~~~~~~~~~~~~~~~~~~~}$};
\node at (2.7, -1.7) {$K_{n-4}$};
\node at (3,2.2) {$-$};
\node at (-0.3,1) {$v_1$};
\node at (2,1.3) {$v_5$};
\node at (4,1.3) {$v_3$};
\node at (6.3,1) {$v_2$};
\draw  [black](0,1)--(-0.3,-0.21);
\draw  [black](0,1)--(2.3,-0.2)--(2,1);
\draw  [black](2,1)--(-0.3,-0.2);
\draw  [black](0,1)--(2,1)--(5.5,-0.08)--(4,1);
\end{tikzpicture}
\caption{The signed graph $\Gamma_{n,\tau}$.}
\label{f2}
\end{figure}
\begin{Lemma}\label{N2}
Let $\Gamma_{n, \tau}$ be the signed graph in Fig.~2 with $n\ge 6$ and $0\le\tau\le n-5$.

(i)
$\lambda_1(\Gamma_{n, \tau})\le \lambda_1(\Gamma_{n,0})= \lambda_1(\Gamma_{n, n-5})$
with equality if and only if $\tau=0, n-5$.

(ii) $\lambda_1(\Gamma_{n,0})>n-4$ and $\lambda_1(\Gamma_{n,0})$ is equal to the largest root of
$\lambda^3-(n-6)\lambda^2-3(n-4)\lambda-n+3=0$.

\end{Lemma}

\begin{proof} First, we show the following claim.

\begin{Claim} \label{AA}
For $0\le \tau\le n-5$, the eigenvalues of
$\Gamma_{n, \tau}$ are $-1$ with multiplicity $n-5$ and the five roots of $f_\tau(\lambda)=0$, where
\begin{align*}
f_\tau(\lambda) &=
\lambda^5-(n-5)\lambda^4-(2n-5)\lambda^3+(2n\tau+3n-2\tau^2-10\tau-15)\lambda^2\\
&\quad +(n\tau+4n-\tau^2-5\tau-15)\lambda+n-3.
\end{align*}
\end{Claim}

\begin{proof}
%Suppose that $\tau\ge 1$. With a proper labeling of the vertices of $\Gamma_{n, \tau}$, we have
%\[
%A(\Gamma_{n, \tau})=
%\begin{pmatrix}
%0 & -1 & 0 & 1 & J_{1\times\tau} & O\\
%-1 & 0 & 1 & 0 & O & O\\
%0 & 1 & 0 & 0 & O & J_{1\times(n-\tau-4)}\\
%1 & 0 & 0 & 0 & J_{1\times\tau} & J_{1\times(n-\tau-4)}\\
%J_{\tau\times1} & O & O & J_{\tau\times1} & (J-I)_{\tau} & J_{\tau\times(n-\tau-4)}\\
%O & O & J_{(n-\tau-4)\times1} & J_{(n-\tau-4)\times1} & J_{(n-\tau-4)\times\tau} & (J-I)_{n-\tau-4}\\
%\end{pmatrix}.
%\]
%and
%\[
%A(\Gamma_{n, 0})=
%\begin{pmatrix}
%0 & -1 & 0 & 1 &  O\\
%-1 & 0 & 1 & 0 &  O\\
%0 & 1 & 0 & 0 & J_{1\times(n-\tau-4)}\\
%1 & 0 & 0 & 0 &  J_{1\times(n-\tau-4)}\\
%O & O & J_{(n-\tau-4)\times1} & J_{(n-\tau-4)\times1} &  (J-I)_{n-\tau-4}\\
%\end{pmatrix}.
%\]
Observe that the $\tau+1$ rows of $I_n+A(\Gamma_{n, \tau})$  corresponding to vertices in $\{v_5\}\cup X$ are equal
and the $n-\tau-4$ rows of $I_n+A(\Gamma_{n, \tau})$  corresponding to vertices in $Y$ are equal. So the rank of $I_n+A(\Gamma_{n, \tau})$ is at most $5$, implying that $-1$ is an eigenvalue of multiplicity at least $n-5$.
Let $V_i=\{v_i\}$ for $i=1,2,3$, $V_4=\{v_5\}\cup X$ and $V_5=Y$.
Then the quotient matrix of $A(\Gamma_{n,\tau})$ with respect to the equitable partition $V(\Gamma_{n,\tau})=V_1\cup \dots\cup V_5$  is
\[
\begin{pmatrix}
0 & -1 & 0 &  \tau+1 & 0\\
-1 & 0 & 1 &  0 & 0\\
0 & 1 & 0 &  0 & n-\tau-4\\
1 & 0 & 0 &  \tau & n-\tau-4\\
0 & 0 & 1 &  \tau+1 & n-\tau-5\\
\end{pmatrix},
\]
whose characteristic polynomial is
$f_\tau(\lambda)$.
As $f_\tau(-1)=(\tau+1)(n-\tau-4)\ne 0$, $-1$ is not a root of $f_\tau(\lambda)=0$.
By Lemma \ref{QM}, the eigenvalues of
$\Gamma_{n, \tau}$ are $-1$ with multiplicity $n-5$ and the five roots of $f_\tau(\lambda)=0$.
\end{proof}

By Claim \ref{AA}, $\lambda_1(\Gamma_{n,\tau})$ is equal to the largest root  of $f_\tau(\lambda)=0$.

Now we prove (i).
Let $U=\{v_2\}$. Then $\Gamma_{n, n-5}^U$ is isomorphic to $\Gamma_{n,0}$, so $\Gamma_{n, n-5}$ is switching equivalent to $\Gamma_{n,0}$. So $\lambda_1(\Gamma_{n,0})= \lambda_1(\Gamma_{n, n-5})$.
If $n=6$, then  (i) is trivial as $\tau=0,1$. Suppose that $n\ge 7$ and $1\le \tau\le n-6$. It suffices to show that  $\lambda_1(\Gamma_{n, \tau})<\lambda_1(\Gamma_{n,0})$. Note that
\[
f_\tau(\lambda)-f_{\tau-1}(\lambda)=(n-2\tau-4)\lambda(2\lambda+1).
\]
If $n-2\tau-4>0$, then $f_\tau(\lambda)>f_{\tau-1}(\lambda)$ for $\lambda>0$, so
$\lambda_1(\Gamma_{n,\tau})<\lambda_1(\Gamma_{n,\tau-1})<\dots<\lambda_1(\Gamma_{n,0})$.
If $n-2\tau-4<0$, then $f_\tau(\lambda)<f_{\tau-1}(\lambda)$ and so $f_{\tau+1}(\lambda)<f_{\tau}(\lambda)$ for $\lambda>0$, implying that $\lambda_1(\Gamma_{n,\tau})<\lambda_1(\Gamma_{n,\tau+1})<\dots<\lambda_1(\Gamma_{n, n-5})$.
Suppose that $n-2\tau-4=0$. Then $f_\tau(\lambda)=f_{\tau-1}(\lambda)$. In this case
$f_{\tau-1}(\lambda)-f_{\tau-2}(\lambda)=2\lambda(2\lambda+1)>0$
for $\lambda>0$,
so $f_{\tau-1}(\lambda)>f_{\tau-2}(\lambda)$, implying that $\lambda_1(\Gamma_{n,\tau})=\lambda_1(\Gamma_{n,\tau-1})
<\lambda_1(\Gamma_{n,\tau-2})<\dots<\lambda_1(\Gamma_{n,0})$. Now (i) follows.

To prove (ii), note that
\[
f_0(\lambda)=(\lambda^2-\lambda-1)g(\lambda),
\]
where $g(\lambda)=\lambda^3-(n-6)\lambda^2-3(n-4)\lambda-n+3$.
It can be checked that $g(n-4)=-n^2+7n-13<0$ and $g(n-3)=2n-6>0$, so $g(\lambda)=0$ has a root larger than $n-4$, which is also larger than the larger root of $\lambda^2-\lambda-1=0$. So $\lambda_1(\Gamma_{n, 0})>n-4$, and
 $\lambda_1(\Gamma_{n, 0})$ is  equal to the largest root  of $g(\lambda)=0$.
\end{proof}

\begin{Lemma}\label{SE}\cite{Z1}
Two signed graphs with the same underlying graph are switching equivalent if and only if they have the same class of positive cycles.
\end{Lemma}

Stani\'{c}  \cite{Stan} noted that for an eigenvalue $\lambda$ of a signed graph $\Gamma$, there is a switching equivalent signed graph for which the $\lambda$-eigenspace contains an eigenvector whose non-zero entries are of the same sign.

\begin{Lemma}\label{L+}\cite{Stan,SL}
A signed graph $\Gamma$ is switching equivalent a signed graph $\Gamma'$ such that $\lambda_1(\Gamma')$ has a non-negative eigenvector.
\end{Lemma}

The balanced clique number of $\Gamma$, denoted by $\omega_b(\Gamma)$, is the maximum order of a balanced complete subgraph.

\begin{Lemma}\label{wb}\cite{WYQ}
Let $\Gamma$ be a signed graph of order $n$. Then
\[
\lambda_1(\Gamma)\le n\left(1-\frac{1}{\omega_b(\Gamma)}\right).
\]
\end{Lemma}

Let $\Gamma$ be a signed graph on vertices $v_1,\dots, v_n$.  A real vector $\mathbf{x}=(x_1, \dots, x_n)^\top$ is viewed as a function on $\{v_1, \dots, v_n\}$ that maps vertex $v_i$  to $x_i$ , i.e.,
$\mathbf{x}(v_i)=x_i$, for $i=1,\dots, n$.

\begin{Lemma} \cite{BS} \label{FII}  Let $\mathbf{x}=(x_1, \dots, x_n)^\top$  be an eigenvector associated
with the largest eigenvalue of a signed graph $\Gamma$ and let $v_r, v_s$ be fixed vertices of $\Gamma$.
If $x_rx_s\ge 0$, at least one of $x_r,x_s$ is nonzero, and $v_r$ and $v_s$  is are not adjacent (resp. $v_rv_s$  is a negative edge), then for a signed
graph $\Gamma'$ obtained by adding a positive edge $v_rv_s$ (resp. removing  $v_rv_s$ or
reversing its sign) we have $\lambda_1(\Gamma')>\lambda_1(\Gamma)$.
\end{Lemma}

\section{Proof of Theorem~\ref{N1}}

\begin{proof}[Proof of Theorem~\ref{N1}]
For a $\mathcal{C}_{3,4}^-$-free unbalanced signed graph of order $6$, the length of a shortest  negative cycle is $5$ or $6$. We display all  $\mathcal{C}_{3,4}^-$-free unbalanced signed graphs of order $6$ under switching equivalence in Fig.~\ref{f3}, where the largest and the smallest eigenvalues are listed below the corresponding signed graphs, and the last signed graph is just $\Gamma_{6,0}$. It is easy to see that
$\Gamma_{6,0}$ is the unique one with maximum spectral radius, which, by Lemma \ref{N2} (ii), is
 the largest root of the equation $\lambda^3-6\lambda-3=0$.
\begin{figure}[htbp]
\centering
\begin{tikzpicture}
\filldraw [black] (1,2) circle (2pt);
\filldraw [black] (2.5,2) circle (2pt);
\filldraw [black] (1,1) circle (2pt);
\filldraw [black] (2.5,1) circle (2pt);
\filldraw [black] (1.75,0.2) circle (2pt);
\filldraw [black] (1.75,1.1) circle (2pt);
\draw  [black](1,2)--(1,1);
\draw  [red](1,2)--(2.5,2);
\draw  [black](2.5,2)--(2.5,1);
\draw  [black](1,1)--(1.75,0.2);
\draw  [black](2.5,1)--(1.75,0.2);
\node at (1.75, 2.2) {$-$};
\node at (1.75, -0.3) {$1.6180$};
\node at (1.75, -0.8) {$-2$};

\filldraw [black] (-0.5,2) circle (2pt);
\filldraw [black] (-1.7,2) circle (2pt);
\filldraw [black] (0,1.1) circle (2pt);
\filldraw [black] (-2.2,1.1) circle (2pt);
\filldraw [black] (-0.5,0.2) circle (2pt);
\filldraw [black] (-1.7,0.2) circle (2pt);
\draw  [black](-0.5,2)--(0,1.1);
\draw  [red](-0.5,2)--(-1.7,2);
\draw  [black](-1.7,2)--(-2.2,1.1);
\draw  [black](0,1.1)--(-0.5,0.2);
\draw  [black](-1.7,0.2)--(-0.5,0.2);
\draw  [black](-1.7,0.2)--(-2.2,1.1);
\node at (-1.1, 2.2) {$-$};
\node at (-1.1, -0.3) {$1.7321$};
\node at (-1.1, -0.8) {$-1.7321$};

\filldraw [black] (3.5,2) circle (2pt);
\filldraw [black] (5,2) circle (2pt);
\filldraw [black] (3.5,1) circle (2pt);
\filldraw [black] (5,1) circle (2pt);
\filldraw [black] (4.25,0.2) circle (2pt);
\filldraw [black] (4.25,1.1) circle (2pt);
\draw  [black](3.5,2)--(3.5,1);
\draw  [red](3.5,2)--(5,2);
\draw  [black](5,2)--(5,1);
\draw  [black](3.5,1)--(4.25,0.2);
\draw  [black](5,1)--(4.25,0.2);
\draw  [black](4.25,1.1)--(3.5,1);
\node at (4.25, 2.2) {$-$};
\node at (4.25, -0.3) {$1.8608$};
\node at (4.25, -0.8) {$-2.1149$};

\filldraw [black] (6,2) circle (2pt);
\filldraw [black] (7.5,2) circle (2pt);
\filldraw [black] (6,1) circle (2pt);
\filldraw [black] (7.5,1) circle (2pt);
\filldraw [black] (6.75,0.2) circle (2pt);
\filldraw [black] (6.75,1.1) circle (2pt);
\draw  [black](6,2)--(6,1);
\draw  [red](6,2)--(7.5,2);
\draw  [black](7.5,2)--(7.5,1);
\draw  [black](6,1)--(6.75,0.2);
\draw  [black](7.5,1)--(6.75,0.2);
\draw  [black](6.75,1.1)--(6,1);
\draw  [black](6.75,1.1)--(6.75,0.2);
\node at (6.75, 2.2) {$-$};
\node at (6.75, -0.3) {$2.3028$};
\node at (6.75, -0.8) {$-2$};

\filldraw [black] (8.5,2) circle (2pt);
\filldraw [black] (10,2) circle (2pt);
\filldraw [black] (8.5,1) circle (2pt);
\filldraw [black] (10,1) circle (2pt);
\filldraw [black] (9.25,0.2) circle (2pt);
\filldraw [black] (9.25,1.1) circle (2pt);
\draw  [black](8.5,2)--(8.5,1);
\draw  [red](8.5,2)--(10,2);
\draw  [black](10,2)--(10,1);
\draw  [black](8.5,1)--(9.25,0.2);
\draw  [black](10,1)--(9.25,0.2);
\draw  [black](9.25,1.1)--(8.5,1);
\draw  [black](9.25,1.1)--(10,1);
\node at (9.25, 2.2) {$-$};
\node at (9.25, -0.3) {$2.1642$};
\node at (9.25, -0.8) {$-2.3914$};

\filldraw [black] (11,2) circle (2pt);
\filldraw [black] (12.5,2) circle (2pt);
\filldraw [black] (11,1) circle (2pt);
\filldraw [black] (12.5,1) circle (2pt);
\filldraw [black] (11.75,0.2) circle (2pt);
\filldraw [black] (11.75,1.1) circle (2pt);
\draw  [black](11,2)--(11,1);
\draw  [red](11,2)--(12.5,2);
\draw  [black](12.5,2)--(12.5,1);
\draw  [black](11,1)--(11.75,0.2);
\draw  [black](12.5,1)--(11.75,0.2);
\draw  [black](11.75,1.1)--(11,1);
\draw  [black](11.75,1.1)--(11.75,0.2);
\draw  [black](11.75,1.1)--(12.5,1);
\node at (11.75, 2.2) {$-$};
\node at (11.75, -0.3) {$2.6691$};
\node at (11.75, -0.8) {$-2.1451$};
\end{tikzpicture}
\caption{The $\mathcal{C}_{3,4}^-$-free unbalanced signed graphs of order $6$ together with the largest and the smallest eigenvalues.}
\label{f3}
\end{figure}
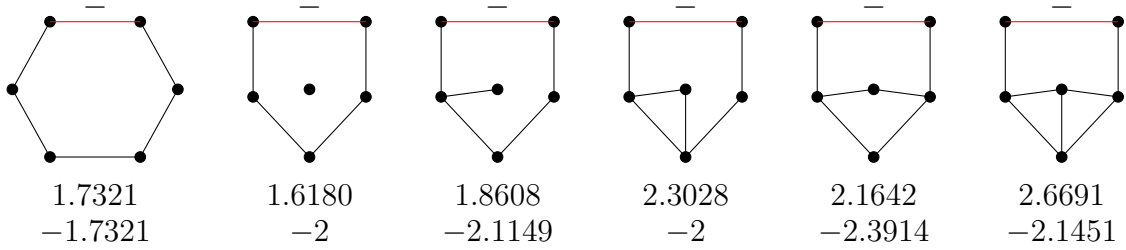

Suppose in the following that $n\ge 7$.
Let $\Gamma=(G,\sigma)$ be a $\mathcal{C}_{3,4}^-$-free  unbalanced signed graph with order $n$ that maximizes the spectral radius.
Since $\Gamma_{n,0}$ is  $\mathcal{C}_{3,4}^-$-free and unbalanced, we have $\rho(\Gamma)\ge\lambda_1(\Gamma_{n,0})>n-4$ by Lemma \ref{N2} (ii). By a direct calculation, we have
$\rho(\Gamma_{7,0})=3.7136>3.7$.

\begin{Claim}\label{C1}
$\rho(\Gamma)=\lambda_1(\Gamma)$.
\end{Claim}

\begin{proof}
Suppose that $\rho(\Gamma)\ne \lambda_1(\Gamma)$.
Since $\rho(\Gamma)=\max\{\lambda_1(\Gamma), -\lambda_n(\Gamma)\}$, we have
$\lambda_1(\Gamma)< -\lambda_n(\Gamma)=\rho(\Gamma)$.
As $\Gamma$ is $\mathcal{C}_3^-$-free, we have $\omega_b(-\Gamma)\le 2$. By Lemma \ref{wb}, we have
\[
\rho(\Gamma)=-\lambda_n(\Gamma)=\lambda_1(-\Gamma)\le n\left(1-\frac{1}{\omega_b(-\Gamma)}\right)\le \frac{n}{2}\le \begin{cases}
n-4 & \mbox{if } n\ge 8,\\
3.7 & \mbox{if } n=7,
\end{cases}
\]
a contradiction.
\end{proof}

By Lemma \ref{L+}, $\Gamma$ is switching equivalent to a signed graph
$\Gamma^*=(G,\sigma^*)$ such that $A(\Gamma^*)$ has a non-negative unit eigenvector $\mathbf{x}=(x_1,  \dots, x_n)^\top$ associated with $\lambda_1(\Gamma^*)$.
By Lemma \ref{SE}, $\Gamma^*$ is a $\mathcal{C}_{3,4}^-$-free  unbalanced signed graph of order $n$.
By Claim \ref{C1}, we have $\rho(\Gamma^*)=\lambda_1(\Gamma^*)=\lambda_1(\Gamma)=\rho(\Gamma)>n-4$.

\begin{Claim}\label{C2}
$\mathbf{x}$ contains at most two zero entries.
\end{Claim}

\begin{proof}
Suppose that this is not true. Assume that $x_i=x_j=x_k=0$ with $1\le i<j<k\le n$. Let $\mathbf{y}$ be the vector obtained from $\mathbf{x}$ by removing the entries $x_i, x_j$ and $x_k$.
It is well known that among graphs with fixed order, the spectral radius is uniquely maximized by the complete graph.
By Rayleigh's principle,
\begin{align*}
\lambda_1(\Gamma^*)&=\mathbf{x}^\top A(\Gamma^*)\mathbf{x}=
\mathbf{y}^\top A(\Gamma^*-\{v_i, v_j, v_k\})\mathbf{y}\\
&\le\lambda_1(\Gamma^*-\{v_i, v_j, v_k\})\le \lambda_1(G-\{v_i, v_j, v_k\})\le
\lambda_1(K_{n-3})=n-4,
\end{align*}
a contradiction.
\end{proof}

\begin{Claim}\label{C3}
$\Gamma^*$ is connected.
\end{Claim}

\begin{proof}
Suppose that $\Gamma^*$ is not connected. Note that $\Gamma^*$ contains at least one negative cycle.  Assume that $v_1v_2$ is an edge in one component with $x_1\ge x_2$ and $v_3$ is a vertex in another component. By Claim \ref{C2}, we have $x_3>0$ or  $x_1>0$.
Let $\Gamma'$ be the signed graph formed by adding a positive edge $v_1v_3$ to $\Gamma^*$.
Evidently, $\Gamma'$ is a $\mathcal{C}_{3,4}^-$-free unbalanced signed graph of order $n$.
By Lemma \ref{FII}, we have $\lambda_1(\Gamma')>\lambda_1(\Gamma^*)$, a contradiction.
\end{proof}

Since $\Gamma^*$ is unbalanced, $\Gamma^*$ contains at least one negative cycle. Let $C:=v_1  \dots  v_rv_1$ be a shortest negative cycle in $\Gamma^*$.  Then $r\ge 5$.

\begin{Claim}\label{zh}
$C$ is an induced cycle.
\end{Claim}

\begin{proof} Suppose that $C$ is not an induced cycle. Then there is a chord $v_sv_t$ of $C$ with $1\le s<t\le r$.  As $C$ is negative,
$v_s$ and $v_t$ separate $C$ into two paths connecting $v_s$ and $v_t$,
 one path $P$ with an even number of negative edges, and the other path $Q$ with an odd number of negative edges. If $v_sv_t$ is an negative edge, then $P$ and $v_sv_t$ form a negative cycle that is shorter than $C$, otherwise, $Q$ and $v_sv_t$ form a negative cycle that is shorter than $C$, which is a contradiction in either case.
\end{proof}

\begin{Claim} \label{z2}
If $v_sv_t$ is  a negative edge of $\Gamma^*$ and it is not an edge of $C$, then $x_s=x_t=0$.
\end{Claim}

\begin{proof}
Construct a signed graph $\Gamma'$ by removing the edge $v_sv_t$ from $\Gamma^*$. As $C$ is a negative cycle in $\Gamma'$,  $\Gamma'$ is   a $\mathcal{C}_{3,4}^-$-free unbalanced signed graph. If  $x_s\ne 0$ or $x_t\ne 0$, then we have by Lemma \ref{FII} that $\lambda_1(\Gamma')>\lambda_1(\Gamma^*)$, a contradiction.
\end{proof}

By Claims \ref{C2} and \ref{z2}, there is at most one negative edge outside $C$.

\begin{Claim} \label{a1}
Any negative edge on $C$ does not lie on a positive $C_3$.
\end{Claim}

\begin{proof}
Suppose that a positive $C_3$ contains a negative edge, say $e=v_1v_2$, on $C$. Let $v_s$ be the vertex of the positive $C_3$ different from $v_1$ and $v_2$. By  Claim \ref{zh},  $v_s\notin V(C)$.
Note that $v_1v_s$ and $v_2v_s$ have opposite signs, say $v_1v_s$ is negative.
By Claims  \ref{z2} and \ref{C2}, $x_1=x_s=0$ and $x_2>0$.
Let $\Gamma'$ be the signed graph formed by removing the edge $v_1v_2$ from $\Gamma^*$. Then $v_2\dots v_rv_1v_sv_2$ is a negative cycle in $\Gamma'$, so $\Gamma'$ is
 obviously $\mathcal{C}_{3,4}$-free and unbalanced. By Lemma \ref{FII}, we have $\lambda_1(\Gamma')>\lambda_1(\Gamma^*)$, a contradiction.
\end{proof}

\begin{Claim} \label{a2}
If there is a negative edge  outside $C$, then it does not lie on a positive $C_4$.
\end{Claim}

\begin{proof} Suppose  that  there is a negative edge $v_sv_t$ outside $C$ but it lies on some positive $C_4$. As $v_sv_t$ is the unique negative edge outside $C$, there is a negative edge, say $v_1v_2$, on $C$
so that it lies on the positive $C_4$ containing $v_sv_t$.

Suppose first that  $\{s, t\}\cap \{1, 2\}\ne \emptyset$, say $s=1$. Then $v_t\notin V(C)$. By Claims \ref{z2} and \ref{C2}, $x_1=x_t=0$ and $x_2>0$.
Let $\Gamma'$ be the signed graph formed by removing the edge $v_1v_2$ from $\Gamma^*$.
Let $u$ be the vertex of the positive $C_4$ different from $v_1, v_2$ and $v_t$.
If $u\in V(C)$, then $u=v_3$, so $v_2v_3$, $v_tv_3$ have the same sign.
Then $\Gamma'$ contains a negative cycle $v_3\dots v_rv_1v_tv_3$ if $u\in V(C)$ and a negative cycle $v_2\dots v_rv_1v_tuv_2$ if $u\notin V(C)$.
So $\Gamma'$ is obviously $\mathcal{C}_{3,4}$-free and unbalanced.
By Lemma \ref{FII}, we have $\lambda_1(\Gamma')>\lambda_1(\Gamma^*)$, a contradiction.

Suppose next that $\{s, t\}\cap \{1, 2\}=\emptyset$. By Claims  \ref{z2} and \ref{C2}, $x_1, x_2>0$. Assume that the positive $C_4$ is $v_1v_2v_tv_sv_1$ (if it is $v_1v_tv_sv_2v_1$, then the proof is similar). Then
$v_1v_s$ and $v_2v_t$ have the same sign.
If $\{v_s, v_t\}\cap V(C)\ne \emptyset$, then  $\{v_s, v_t\}\cap V(C)=\{v_s\}$ with  $s=r$, or $\{v_s, v_t\}\cap V(C)=\{v_t\}$ with  $t=3$.
Let $\Gamma''$ be the signed graph formed by removing the edge $v_1v_2$ from $\Gamma^*$.
If $\{v_s, v_t\}\cap V(C)=\emptyset$, then $\Gamma''$ contains a negative cycle
$v_2\dots v_rv_1v_sv_tv_2$; Otherwise,
$\Gamma''$ contains a negative cycle obtained by replacing $v_rv_1, v_1v_2$ with $v_rv_t, v_tv_2$ in $C$ if $\{v_s, v_t\}\cap V(C)=\{v_s\}$ (resp. by replacing $v_1v_2, v_2v_3$ with $v_1v_s, v_sv_3$ in $C$ if $\{v_s, v_t\}\cap V(C)=\{v_t\}$).
So $\Gamma''$ is   $\mathcal{C}_{3,4}^-$-free and unbalanced.
Then we have By Lemma \ref{FII} that
$\lambda_1(\Gamma'')>\lambda_1(\Gamma^*)$, a contradiction.
\end{proof}

\begin{Claim}\label{a3}
$C$ contains exactly one negative edge.
\end{Claim}

\begin{proof}
Suppose that $C$ contains more than one negative edges. Choose two negative edges $v_1v_2$ and $v_iv_{i+1}$ on $C$, where $2\le i\le r$. If $i=r$, then $v_{i+1}=v_1$.
Let $\widetilde\Gamma$ be a signed graph from $\Gamma^*$ by reversing the sign of $v_1v_2$ and $v_iv_{i+1}$ from $\Gamma^*$.

We claim that $\widetilde\Gamma$ is $\mathcal{C}_{3,4}^-$-free.

Suppose that there is a negative $C_3$ in $\widetilde\Gamma$. Then it contains one of the positive edges $v_1v_2$ and $v_iv_{i+1}$. Claim \ref{zh}, it contains exactly one of the positive edges $v_1v_2$ and $v_iv_{i+1}$,  say  $v_1v_2$.
Then reversing the sign of $v_1v_2$ in this negative $C_3$ in $\widetilde\Gamma$, we obtain a positive $C_3$ containing the negative $v_1v_2$ in $\Gamma^*$, contradicting Claim \ref{a1}.

Suppose there is a negative $C_4$ in $\widetilde\Gamma$. Then it contains one of $v_1v_2$ and $v_iv_{i+1}$. If it contains both $v_1v_2$ and $v_iv_{i+1}$, then $i=2$ or $i=r$ by Claim \ref{zh}, so by reversing the sign of $v_1v_2$ and $v_iv_{i+1}$ in this negative $C_4$ in $\widetilde\Gamma$, we obtain a negative $C_4$ in $\Gamma^*$, which is impossible.
So it contains exactly one of $v_1v_2$ and $v_iv_{i+1}$, say $v_1v_2$.
Assume that it is $v_1v_sv_tv_2v_1$. By  reversing the sign of $v_1v_2$ in this negative $C_4$ of $\widetilde\Gamma$, we obtain a positive $C_4$ of $\Gamma^*$ containing the negative edge $v_1v_2$. By Claim \ref{a2}, it does not contain any negative edge outside $C$, so it contains the negative edge $v_2v_3$ or $v_1v_r$, i.e., $t=3$ or $s=r$, say $t=3$.
By Claim \ref{C2}, one of $x_1$, $x_2$ and $x_3$ is nonzero.
Construct a signed graph $\Gamma'$ by removing the edge $v_1v_2$ from $\Gamma^*$ if $x_1\ne0$ or $x_2\ne0$, and by removing an edge $v_2v_3$ if $x_3\ne0$. Note that replacing two edges $v_1v_2$ and $v_2v_3$ on $C$ by two edges $v_1v_s$ and $v_sv_3$, we get a negative cycle in $\Gamma'$, so $\Gamma'$ is $\mathcal{C}_{3,4}^-$-free and unbalanced. By Lemma \ref{FII}, we have $\lambda_1(\Gamma')>\lambda_1(\Gamma^*)$, a contradiction.

It follows that $\widetilde\Gamma$ is $\mathcal{C}_{3,4}^-$-free, as desired.

As the cycle obtained from $C$ by reversing the sign of $v_1v_2$ and $v_iv_{i+1}$ is a negative cycle in  $\widetilde\Gamma$, $\widetilde\Gamma$ is unbalanced.
By Rayleigh's principle, we have
\[
\lambda_1(\widetilde\Gamma)-\lambda_1(\Gamma^*)\ge\mathbf{x}^T(A(\widetilde\Gamma)-A(\Gamma^*))\mathbf{x}
=4(x_1x_2+x_ix_{i+1})\ge0,
\]
so $\lambda_1(\widetilde\Gamma)=\lambda_1(\Gamma^*)$, implying that $\mathbf{x}$ is an eigenvector of
$A(\widetilde\Gamma)$ associated with $\lambda_1(\widetilde\Gamma)$.
If $i=2,r$, say $i=2$, then
\[
\lambda_1(\widetilde\Gamma)x_{1}=\lambda_1(\Gamma^*)x_{1}+2x_2
\]
and
\[
\lambda_1(\widetilde\Gamma)x_{2}=\lambda_1(\Gamma^*)x_{2}+2(x_1+x_3),
\]
so $x_2=0$ and $x_1+x_3=0$, implying that $x_1=x_2=x_3=0$,  contradicting Claim \ref{C2}.
If $3\le i\le r-1$, then $x_1x_2=x_ix_{i+1}=0$, so assuming $x_1=x_i=0$, we have
\[
\lambda_1(\widetilde\Gamma)x_{i}=\lambda_1(\Gamma^*)x_{i}+2x_{i+1},
\]
so $x_{i+1}=0$, contradicting Claim \ref{C2}.
\end{proof}

By Claim \ref{a3}, $C$ contains exactly one negative edge, say $v_1v_2$.

\begin{Claim}\label{C4}
There is no negative edge of $\Gamma^*$ outside $C$.
\end{Claim}

\begin{proof} Suppose that this is not true. Then
there is a negative edge, say $v_sv_t$,  of $\Gamma^*$  outside $C$.
By Claim \ref{z2}, $x_s=x_t=0$.
 By Claims \ref{z2} and \ref{C2}, $v_sv_t$ is the unique negative edge of $\Gamma^*$  that is not on $C$.
By Claim \ref{a3},  $v_1v_2$ and $v_sv_t$ are the only negative edges of $\Gamma^*$.

Suppose that $\{s,t\}\cap \{1,2\}\ne\emptyset$, say $s=1$. Then $v_t$ lies outside $C$.
By Claim \ref{a1}, $v_2v_t\notin E(\Gamma^*)$. Let $\Gamma'$ be a signed graph formed  by adding a positive edge $v_2v_t$ to $\Gamma^*$.
Evidently, there is no negative $C_3$ containing $v_2v_t$ in $\Gamma'$. If there is a negative $C_4$ containing $v_2v_t$  in $\Gamma'$, then it contains one of $v_1v_2$ and $v_1v_t$, say $v_1v_2$, then it is $v_1v_2v_tuv_1$ for some common neighbor $u$ of $v_1$ and $v_t$ in $\Gamma^*$, so $v_1v_tuv_1$ is a negative $C_3$ in $\Gamma^*$, a contradiction. So $\Gamma'$ is $\mathcal{C}_{3,4}^-$-free,  and as $C$ is a negative cycle of  $\Gamma'$, it is  unbalanced. By Claim \ref{C2}, $x_2>0$. Then
by Lemma \ref{FII}, we have $\lambda_1(\Gamma')>\lambda_1(\Gamma^*)$, a contradiction. Thus $\{s,t\}\cap \{1,2\}=\emptyset$.
By Claim \ref{C2}, $x_1, x_2>0$.

\noindent{\bf Case 1.} $\{v_s, v_t\}\cap V(C)\ne\emptyset$.

Assume that $v_s\in V(C)$ with $\lceil\frac{r+1}{2}\rceil+1\le s\le r$. Then $v_t\notin V(C)$.
By Claim \ref{zh}, $C$ is an induced cycle.

\noindent{\bf Case 1.1.}   $s\le r-2$.

In this case, $r\ge 7$.
By Claim \ref{zh}, $v_2v_s\not\in E(\Gamma^*)$. Let $\widetilde\Gamma$ be a signed graph formed by adding a positive $v_2v_s$. It is obvious that $\widetilde\Gamma$ is unbalanced as $C$ is also the negative cycle in $\widetilde\Gamma$.

We claim that $\widetilde\Gamma$ is $\mathcal{C}_{3,4}^-$-free.
If there is a negative $C_3$ in $\widetilde\Gamma$, then it contains $v_2v_s$ and exactly one of $v_1v_2$ and $v_sv_t$, so  $v_1v_s\in E(\Gamma^*)$ or $v_2v_t\in E(\Gamma^*)$. By Claim \ref{zh}, $v_1v_s\not\in E(\Gamma^*)$. If $v_2v_t\in E(\Gamma^*)$, then we would have a negative cycle
$v_2\dots v_sv_tv_2$ that is shorter than $C$ in $\Gamma^*$, a contradiction. So $\widetilde\Gamma$ is $\mathcal{C}_3^-$-free.
Suppose that there is a negative $C_4$ in $\widetilde\Gamma$. Then it contains $v_2v_s$ and exactly one of $v_1v_2$ and $v_sv_t$. Suppose that this negative $C_4$ of  $\widetilde\Gamma$ contains $v_1v_2$. Then it is
$v_1v_2v_suv_1$ for some common neighbor $u$ of $v_1$ and $v_s$. By Claim \ref{zh} and $s\le r-2$, $u\not\in V(C)$, so
$v_1\dots v_suv_1$ is a negative cycle in $\Gamma^*$ that is shorter than $C$, which is a contradiction.
Suppose that this negative $C_4$ of  $\widetilde\Gamma$ contains $v_sv_t$. Then it is $v_2v_sv_twv_2$ for some common neighbor $w$ of $v_2$ and $v_t$.
If $w\in V(C)$, then $w=v_3$. So $\Gamma^*$ contains a negative cycle $v_3\dots v_sv_tv_3$ if $w\in V(C)$, and $v_2\dots v_sv_twv_2$ if $w\not\in V(C)$, which are shorter than $C$, a contradiction.
It follows that $\widetilde\Gamma$ is $\mathcal{C}_{3,4}^-$-free.

As $x_2>0$, we have by Lemma \ref{FII} that $\lambda_1(\widetilde\Gamma)>\lambda_1(\Gamma^*)$, a contradiction.

\noindent{\bf Case 1.2.}  $s=r-1$.

Note that $v_1v_t\notin E(\Gamma^*)$, as otherwise $v_1v_rv_{r-1}v_tv_1$ is a negative $C_4$ in $\Gamma^*$, a contradiction.
Let $\widetilde\Gamma$ be a signed graph formed by removing a negative edge $v_{r-1}v_t$ and adding a positive $v_1v_t$ from $\Gamma^*$.

We claim that $\widetilde\Gamma$ is $\mathcal{C}_{3,4}^-$-free.
If there is a negative $C_3$ in $\widetilde\Gamma$, then it contains both $v_1v_t$ and $v_1v_2$, so  $v_2v_t\in E(\Gamma^*)$, implying that there is  a negative cycle
$v_2\dots v_{r-1}v_tv_2$ that is shorter than $C$ in $\Gamma^*$, a contradiction. So $\widetilde\Gamma$ is $\mathcal{C}_3^-$-free.
Suppose that there is a negative $C_4$ in $\widetilde\Gamma$. Then it contains both $v_1v_t$ and $v_1v_2$, so it is $v_tv_1v_2uv_t$ for some common neighbor $u$ of $v_2$ and $v_t$ in $\Gamma^*$. If $u\in V(C)$, then $u=v_3$, so $v_3\dots v_{r-1}v_tv_3$  is negative cycle  shorter than $C$ in $\Gamma^*$, a contradiction. So $u\notin V(C)$. Let $\Gamma''$ be a signed graph formed by removing $v_1v_2$ from $\Gamma^*$.
As $v_2\dots v_{r-1}v_tuv_2$ is a negative cycle in $\Gamma''$,    $\Gamma''$ is $\mathcal{C}_{3,4}^-$-free and unbalanced. By Lemma \ref{FII}, we have  $\lambda_1(\Gamma'')>\lambda_1(\Gamma^*)$,  a contradiction.
It follows that $\widetilde\Gamma$ is $\mathcal{C}_{3,4}^-$-free.
Evidently, $\widetilde\Gamma$ is unbalanced.
By Rayleigh's principle, we have
\[
\lambda_1(\widetilde\Gamma)-\lambda_1(\Gamma^*)\ge\mathbf{x}^T(A(\widetilde\Gamma)
-A(\Gamma^*))\mathbf{x}=2x_t(x_1+x_{r-1})=0,
\]
so $\lambda_1(\widetilde\Gamma)=\lambda_1(\Gamma^*)$. Then $\lambda_1(\widetilde\Gamma)x_t=\lambda_1(\Gamma^*)x_t+x_1+x_{r-1}$, which implies that $x_1=0$, a contradiction.

\noindent{\bf Case 1.3.}  $s=r$.

%Suppose that $v_2v_t\in E(\Gamma^*)$.  Let $\Gamma''$ be a signed graph formed by removing the edge $v_1v_2$ from $\Gamma^*$ . As $v_2\dots v_rv_tv_2$ is a negative cycle in $\Gamma''$,
%$\Gamma''$  is $\mathcal{C}_{3,4}^-$-free and  unbalanced. By Lemma \ref{FII}, we have
%$\lambda_1(\Gamma'')>\lambda_1(\Gamma^*)$,  a contradiction. Thus,

By Claim \ref{a2},
$v_2v_t\notin E(\Gamma^*)$.
Let $\widetilde\Gamma$ be a signed graph formed by adding a positive $v_2v_t$ from $\Gamma^*$.

If there is a negative $C_3$ in $\widetilde\Gamma$, then it contains $v_2v_t$ and exactly one of $v_1v_2$ and $v_rv_t$, so $v_1v_t\in E(\Gamma^*)$ or $v_2v_r\in E(\Gamma^*)$.
But $v_1v_t\in E(\Gamma^*)$ leads to a negative $C_3:=v_1v_rv_tv_1$ in $\Gamma^*$, and $v_2v_r\in E(\Gamma^*)$ contradicts Claim \ref{zh}. So $\widetilde\Gamma$ is $\mathcal{C}_3^-$-free.
Suppose that there is a negative $C_4$ in $\widetilde\Gamma$. Then it contains $v_2v_t$ and exactly one of $v_1v_2$ and $v_rv_t$. If it contains $v_1v_2$, then it is
$v_1v_2v_tuv_1$ for some common neighbor $u$ of $v_1$ and $v_t$ in $\Gamma^*$ with $u\ne v_r$, as  $u\not\in V(C)$ by Claim \ref{zh}, $v_1v_rv_tuv_1$ is a negative $C_4$ in $\Gamma^*$, a contradiction.
Similarly, we have a contradiction if it contains $v_rv_t$.
It follows that $\widetilde\Gamma$ is $\mathcal{C}_{3,4}^-$-free.
Evidently, $\widetilde\Gamma$ is unbalanced.
As $x_1, x_2>0$, we have by Lemma \ref{FII}, $\lambda_1(\widetilde\Gamma)>\lambda_1(\Gamma^*)$, a contradiction.

\noindent{\bf Case 2.} $\{v_s, v_t\}\cap V(C)=\emptyset$.

By Claim \ref{a2}, $\{v_1v_s, v_2v_t\}\not\subseteq E(\Gamma^*)$, and
$\{v_1v_t, v_2v_s\}\not\subseteq E(\Gamma^*)$.
Assume that $v_1v_s\notin E(\Gamma^*)$. There are two possibilities: $v_1v_t\notin E(\Gamma^*)$ or $v_2v_s\notin E(\Gamma^*)$.

If  $v_1v_t\notin E(\Gamma^*)$, then $\{v_2v_t, v_2v_s\}\not\subseteq E(\Gamma^*)$, as otherwise, $v_2v_sv_tv_2$ would be  a negative $C_3$. So
one of the following must hold:
\begin{enumerate}
\item[(i)] $v_2v_s\notin E(\Gamma^*)$ and $v_2v_t\notin E(\Gamma^*)$,

\item[(ii)] $v_2v_t\in E(\Gamma^*)$ and $v_2v_s\notin E(\Gamma^*)$,

\item[(iii)] $v_2v_s\in E(\Gamma^*)$ and $v_2v_t\notin  E(\Gamma^*)$.
\end{enumerate}

Suppose first that (i) holds.
Suppose that $u\in N_{\Gamma^*}(v_1)\cap N_{\Gamma^*}(v_t)\ne\emptyset$ and $w\in N_{\Gamma^*}(v_2)\cap N_{\Gamma^*}(v_s)\ne\emptyset$. Note that $u\ne w$, as otherwise, $v_1v_2uv_1$ is a negative $C_3$ in $\Gamma^*$, a contradiction.
Construct a signed graph $\Gamma''$ by removing $v_1v_2$ from $\Gamma^*$.
As $v_2\dots v_rv_1uv_tv_swv_2$ is a negative cycle in $\Gamma''$,
it is easily to see that $\Gamma''$ is $\mathcal{C}_{3,4}^-$-free and unbalanced. By Lemma \ref{FII}, we have  $\lambda_1(\Gamma'')>\lambda_1(\Gamma^*)$, a contradiction. It follows that  $N_{\Gamma^*}(v_1)\cap N_{\Gamma^*}(v_t)=\emptyset$ or $N_{\Gamma^*}(v_2)\cap N_{\Gamma^*}(v_s)=\emptyset$.
Assume that $N_{\Gamma^*}(v_1)\cap N_{\Gamma^*}(v_t)=\emptyset$.

Let $\widehat\Gamma$ be a signed graph formed by removing the edge $v_sv_t$ and adding a positive edge $v_2v_t$ from $\Gamma^*$.
If there is a negative $C_3$ in $\widehat\Gamma$, then it contains both $v_2v_t$ and $v_1v_2$, so $v_1v_t\in E(\Gamma^*)$, a contradiction. %So $\widehat\Gamma$ is $\mathcal{C}_3^-$-free.
If  there is a negative $C_4$ in $\widehat\Gamma$, then it contains both $v_2v_t$ and $v_1v_2$, so  $v_1$ and $v_t$ have a common neighbor in $\Gamma^*$, a contradiction.
It follows that $\widehat\Gamma$ is $\mathcal{C}_{3,4}^-$-free.
Evidently, $\widetilde\Gamma$ is unbalanced.
Note that
\[
\lambda_1(\widehat\Gamma)-\lambda_1(\Gamma^*)\ge2x_t(x_2+x_s)=0.
\]
Then $\lambda_1(\widehat\Gamma)=\lambda_1(\Gamma^*)$, implying that $\lambda_1(\widehat\Gamma)x_t=\lambda_1(\Gamma^*)x_t+x_2+x_s$, so $x_2=0$,
a contradiction.

Suppose next that (ii) holds.
Let $\widetilde\Gamma$ be a signed graph formed by adding a positive edge $v_1v_s$ from $\Gamma^*$.
If there is a negative $C_3$ in $\widetilde\Gamma$, then it contains $v_1v_s$ and exactly one of $v_1v_2$ and $v_sv_t$, so $v_2v_s\in E(\Gamma^*)$ or $v_1v_t\in E(\Gamma^*)$, a contradiction.
If there is a negative $C_4$ in $\widetilde\Gamma$, then it contains $v_1v_s$ and exactly one of $v_1v_2$ and $v_sv_t$, say $v_1v_2$.  Then $v_2$ and $v_s$ have a common neighbor $u$ in $\Gamma^*$ with $n\ne v_t$, so $v_2uv_sv_tv_2$ is a negative $C_4$ in $\Gamma^*$, a contradiction.
%Similarly, we have a contradiction in the latter case.
It follows that $\widetilde\Gamma$ is $\mathcal{C}_{3,4}^-$-free.
Evidently, $\widetilde\Gamma$ is unbalanced.
As $x_1>0$, we have by Lemma \ref{FII}, $\lambda_1(\widetilde\Gamma)>\lambda_1(\Gamma^*)$, a contradiction.

Suppose that (iii) holds.
Let $\widetilde\Gamma$ be a signed graph formed by adding a positive $v_1v_t$ from $\Gamma^*$.
If there is a negative $C_3$ in $\widetilde\Gamma$, then it contains $v_1v_t$ and exactly one of $v_1v_2$ and $v_sv_t$, so $v_2v_t\in E(\Gamma^*)$ or $v_1v_s\in E(\Gamma^*)$, a contradiction.
Suppose that there is a negative $C_4$ in $\widetilde\Gamma$. Then it contains $v_1v_t$ and exactly one of $v_1v_2$ and $v_sv_t$, say $v_1v_2$.
 %In the former case,
 Then $v_2$ and $v_t$ have a common neighbor $v\ne v_1,v_s$, so $v_2vv_tv_sv_2$ is a negative $C_4$ in $\Gamma^*$, a contradiction.
 %Similarly, we have a contradiction in the latter case.
It follows that $\widetilde\Gamma$ is $\mathcal{C}_{3,4}^-$-free.
Evidently, $\widetilde\Gamma$ is unbalanced.
As $x_1>0$, we have by Lemma \ref{FII}, $\lambda_1(\widetilde\Gamma)>\lambda_1(\Gamma^*)$, a contradiction.

Suppose that  $v_2v_s\notin E(\Gamma^*)$. Then it is impossible that both $v_tv_1\in E(\Gamma^*)$ and $v_tv_2\in E(\Gamma^*)$. By similar argument as above, we also have a contradiction.
\end{proof}

%\begin{Claim}\label{C5}
%$\Gamma^*$ contains exactly one negative edge.
%\end{Claim}
%
%\begin{proof}  Suppose that $\Gamma^*$ contains more than one negative edge.
%Since $\Gamma^*$ is unbalanced, $\Gamma^*$ contains at least three negative edges. By Claim \ref{C4}, there are at least three negative edges on $C$.  Choose two  negative edges $v_iv_j$ and $v_kv_\ell$ on $C$ so that $\{i,j\}\cap \{k,\ell\}=\emptyset$.
%Let $\Gamma'$ be the signed graph obtained from $\Gamma^*$ by reversing the sign of $v_iv_j$ and $v_kv_\ell$. As $\Gamma^*$ is $\mathcal{C}_{3,4}^-$-free, $\Gamma'$ is also a $\mathcal{C}_{3,4}^-$-free unbalanced signed graph. Then
%\[
%\lambda_1(\Gamma')-\lambda_1(\Gamma^*)\ge\mathbf{x}^T(A(\Gamma')-A(\Gamma^*))\mathbf{x}
%=4x_ix_j+4x_kx_\ell\ge0,
%\]
%so $\lambda_1(\Gamma')=\lambda_1(\Gamma^*)$.
%Then $x_ix_j=x_kx_\ell=0$. Assume that $x_i=x_k=0$.   As
%\[
%\lambda_1(\Gamma')x_1=\lambda_1(\Gamma^*)x_i+2x_j,
%\]
%we have $x_j=0$, contradicting Claim \ref{C2}.
%\end{proof}

By Claim  \ref{C4}, $v_1v_2$ is the unique negative edge of $\Gamma^*$, which lies on $C$.

\begin{Claim}\label{C6}
$x_i>0$ for $3\le i\le n$.
\end{Claim}

\begin{proof} By Claim \ref{C3}, $d_{\Gamma^*}(v_i)\ge 1$ for $i=1,\dots, n$.

First, we show that $x_i>0$ for any $3\le i\le n$ with $d_{\Gamma^*}(v_i)\ge 2$.
Otherwise, $x_i=0$ for some $i$ with $3\le i\le n$ and $d_{\Gamma^*}(v_i)\ge 2$.
By Claim \ref{C2}, one entry of $\mathbf{x}$ at some vertex of  $N_{\Gamma^*}(v_i)$ is nonzero.
As $v_1v_2$ is the unique negative edge of $\Gamma^*$, all edges of $\{v_sv_i : v_s\in N_{\Gamma^*}(v_i)\}$ are positive, so
$\sum\limits_{v_s\in N_{\Gamma^*}(v_i)}\sigma(v_sv_i)x_s>0$. However, from $A(\Gamma^*)\mathbf{x}=\lambda_1(\Gamma^*)\mathbf{x}$, we have $\sum\limits_{v_s\in N_{\Gamma^*}(v_i)}\sigma(v_sv_i)x_s=\lambda_1(\Gamma^*)x_i=0$, a contradiction.

Next, we show that $x_i>0$ for any $r+1\le i\le n$ with $d_{\Gamma^*}(v_i)=1$.
Otherwise, $x_i=0$ for some $i$ with $r+1\le i\le n$ and $d_{\Gamma^*}(v_i)=1$. Let
\[
j=\begin{cases}
r-1 & \mbox{if $N_{\Gamma^*}(v_i)\ne \{v_{r-1}\}$},\\
3 & \mbox{otherwise}.
\end{cases}
\]
By the above argument, $x_j>0$.
Construct a signed graph $\Gamma'$ by adding a positive $v_i v_j$ from $\Gamma^*$ .
As $v_1v_2$ is the unique negative edge of $\Gamma^*$,
it may be easily checked that $\Gamma'$ is  $\mathcal{C}_{3,4}^-$-free and unbalanced.
By Lemma \ref{FII}, we have $\lambda_1(\Gamma')>\lambda_1(\Gamma^*)$, a contradiction.
\end{proof}

\begin{Claim}\label{C7}
$C$ is $5$-cycle, i.e., $r=5$.
\end{Claim}

\begin{proof}
Note that $r\ge5$. If $r\ge6$, then we construct a signed graph $\Gamma'$ by adding a positive $v_3v_{r-1}$, which may be easily checked is $\mathcal{C}_{3,4}^-$-free and unbalanced. By Lemma \ref{FII} and Claim \ref{C6}, $\lambda_1(\Gamma')>\lambda_1(\Gamma^*)$, a contradiction.
\end{proof}

By Claims \ref{zh} and \ref{C7}, $C=v_1v_2v_3v_4v_5v_1$ and it is an induced cycle in $\Gamma^*$. Note that it is a shortest negative cycle containing  the unique negative edge $v_1v_2$ of $\Gamma^*$.
Let $X=N_{\Gamma^*}(v_1)\setminus\{v_2, v_5\}$, $X'=N_{\Gamma^*}(v_2)\setminus\{v_1, v_3\}$ and $Y=V(\Gamma^*)\setminus(X\cup X'\cup\{v_1, v_2, v_3, v_5\})$.
Since $\Gamma^*$ is $\mathcal{C}_{3,4}^-$-free, $X\cap X'=\emptyset$ and for any $u\in X$, $w\in X'$,
$u$ and $w$, $u$ and $v_3$, $w$ and $v_5$ are not adjacent in $\Gamma^*$. Let $|X|=\tau$ and $|X'|=\tau'$. Then $|Y|=n-\tau-\tau'-4$. By Lemma \ref{FII} and Claim \ref{C6}, we see that
$\Gamma^*[X\cup \{v_5\}\cup Y]=(K_{n-\tau'-3}, +)$ and  $\Gamma^*[X'\cup \{v_3\}\cup Y]=(K_{n-\tau-3}, +)$. In particular, $v_5$ is adjacent to all vertices in $X\cup Y$ and $v_3$ is adjacent to all vertices in $X'\cup Y$. The structure of $\Gamma^*$ is displayed in Fig.~\ref{f4}.
%where a thick line between a vertex and a vertex of a set means that the vertex is adjacent to every vertex of the set.

\begin{figure}[htbp]
\centering
\begin{tikzpicture}
\draw [red](0,1) .. controls (2,3) and (4,3).. (6,1);
\draw  [black](0,1)--(2,1)--(2.18,-0.17)--(4,1)--(6,1);
\filldraw [black] (0,1) circle (2pt);
\filldraw [black] (2,1) circle (2pt);
\filldraw [black] (4,1) circle (2pt);
\filldraw [black] (6,1) circle (2pt);
\node at (0.6, -0.4) {$X$};
\node at (3, -0.4) {$Y$};
\node at (5.4, -0.4) {$X'$};
\draw (0.6,-0.4) ellipse (0.9 and 0.5);
\draw (3,-0.4) ellipse (0.9 and 0.5);
\draw (5.4,-0.4) ellipse (0.9 and 0.5);
\node at (1.8, -1.2) {$\underbrace{~~~~~~~~~~~~~~~~~}$};
\node at (1.8, -1.7) {$K_{n-\tau'-4}$};
\node at (4.2, -1.2) {$\underbrace{~~~~~~~~~~~~~~~~~}$};
\node at (4.2, -1.7) {$K_{n-\tau-4}$};
\node at (3,2.2) {$-$};
\node at (-0.3,1) {$v_1$};
\node at (2,1.3) {$v_5$};
\node at (4,1.3) {$v_3$};
\node at (6.3,1) {$v_2$};
\draw  [black](0,1)--(-0.15,-0.12)--(2,1);
\draw  [black](0,1)--(1.45,-0.22)--(2,1);
\draw  [black](2,1)--(3.8,-0.16)--(4,1);
\draw  [black](4,1)--(4.57,-0.2)--(6,1);
\draw  [black](4,1)--(6.19,-0.14)--(6,1);
\end{tikzpicture}
\caption{The structure of $\Gamma^*$.}
\label{f4}
\end{figure}
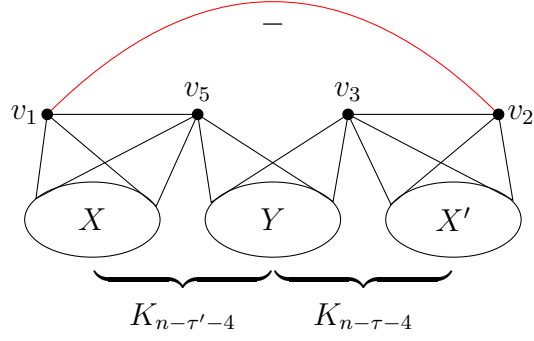
Note that
\[
\lambda_1(\Gamma^*)x_1=-x_2+x_5+\sum\limits_{v_i\in X}x_i
\]
and
\[
\lambda_1(\Gamma^*)x_2=-x_1+x_3+\sum\limits_{v_j\in X'}x_j.
\]
So
\[
(\lambda_1(\Gamma^*)-1)(x_1-x_2)+x_3-x_5=\sum\limits_{v_i\in X}x_i-\sum\limits_{v_j\in X'}x_j.
\]
Similarly, we have
\[
\lambda_1(\Gamma^*)x_5=x_1+\sum\limits_{v_i\in X}x_i+\sum\limits_{v_k\in Y}x_k
\]
and
\[
\lambda_1(\Gamma^*)x_3=x_2+\sum\limits_{v_j\in X'}x_j+\sum\limits_{v_k\in Y}x_k,
\]
so
\[
\lambda_1(\Gamma^*)(x_5-x_3)+x_2-x_1=\sum\limits_{v_i\in X}x_i-\sum\limits_{v_j\in X'}x_j.
\]
Thus
\[
\lambda_1(\Gamma^*)(x_1-x_2)+(\lambda_1(\Gamma^*)+1)(x_3-x_5)=0.
\]
Assume that $x_1\ge x_2$. Then $x_3\le x_5$.

\begin{Claim}\label{C9}
$X'=\emptyset$.
\end{Claim}

\begin{proof}
Suppose that $X'\ne\emptyset$. Let $X'=\{v_{j_1}, \dots, v_{j_{\tau'}}\}$ with $\tau'\ge 1$. Construct a signed graph $\Gamma'$ by removing the (positive) edges  $v_2v_{j_i}$, $v_3v_{j_i}$ and adding the positive edges $v_1v_{j_i}$ and $v_5v_{j_i}$ for all $1\le i\le \tau'$, which can be easily checked that is $\mathcal{C}_{3,4}^-$-free and unbalanced. Then
\[
\lambda_1(\Gamma')-\lambda_1(\Gamma^*)\ge\mathbf{x}^T(A(\Gamma')-A(\Gamma^*))\mathbf{x}=\sum\limits_{1\le i\le \tau'}2x_{v_{j_i}}(x_1-x_2+x_5-x_3)\ge0,
\]
so $\lambda_1(\Gamma')=\lambda_1(\Gamma^*)$, implying that
\[
\lambda_1(\Gamma')x_1=\lambda_1(\Gamma^*)x_1+\sum\limits_{1\le i\le \tau'}x_{v_{j_i}},
\]
so $x_{v_{j_i}}=0$ for $1\le i\le \tau'$, contradicting Claim \ref{C6}.
\end{proof}

By Claim \ref{C9}, $X'=\emptyset$. Then $\Gamma^*$ is $\Gamma_{n, \tau}$ for some $\tau$ with $0\le\tau\le n-5$. By Lemma \ref{N2}, $\tau=0, n-5$.
As $\Gamma_{n, n-5}$ is switching equivalent to $\Gamma_{n,0}$,
$\Gamma^*$ is switching equivalent to $\Gamma_{n,0}\cong \Gamma_n$.  Therefore,
$\Gamma$ is switching isomorphic to $\Gamma_{n,0}$.
By Lemma \ref{N2} (ii),  $\lambda_1(\Gamma_{n,0})$ is equal to the largest root of
$\lambda^3-(n-6)\lambda^2-3(n-4)\lambda-n+3=0$.
\end{proof}

\bigskip

\noindent {\bf Acknowledgement.}
%The authors thank the referees for helpful and constructive comments and suggestions.
This work was supported by National Natural Science Foundation of China (No. 12071158).

\end{document}